\documentclass[1p,times]{elsarticle}

\usepackage{amssymb}
\usepackage{graphicx}
\usepackage[cmex10]{amsmath}
\usepackage{array}
\usepackage{mdwmath}
\usepackage{xurl}
\usepackage{mdwtab}
\usepackage{fixltx2e}
\usepackage{graphics} % for pdf, bitmapped graphics files
\usepackage{epsfig} % for postscript graphics files
\usepackage{times} % assumes new font selection scheme installed

\usepackage{url}
\usepackage{lscape}
\usepackage{color}
\usepackage{epsfig}

\usepackage[utf8x]{inputenc}
\usepackage{graphicx}
\usepackage[cmex10]{amsmath}
\usepackage{array}
\usepackage{graphics}

\usepackage{amsthm}

\usepackage{url}
\usepackage{lscape}
\usepackage{color}
\usepackage{epsfig}
\newtheorem{theorem}{Theorem}
\newtheorem{remark}{Remark}

\newtheorem{problem}{Problem}

\newtheorem{assumption}{Assumption}

\newtheorem*{lemma-non}{Lemma}

\usepackage{balance}

\journal{Engineering Applications of Artificial Intelligence}

\begin{document}

\begin{frontmatter}

\title{Provably-Stable Neural Network-Based Control of Nonlinear Systems\tnoteref{label1}}

\tnotetext[label1]{This research has been supported by WSU Voiland College of Engineering and Architecture through a start-up package to M. Hosseinzadeh.}

\author{Anran Li}
\ead{anran.li@wsu.edu}

\author{John P. Swensen}
\ead{john.swensen@wsu.edu}

\author{Mehdi Hosseinzadeh\corref{cor}}
\ead{mehdi.hosseinzadeh@wsu.edu}

\cortext[cor]{Corresponding author.}

\address{School of Mechanical and Materials Engineering, Washington State University, Pullman, WA 99164, USA}

\begin{abstract}
In recent years, Neural Networks (NNs) have been employed to control nonlinear systems due to their potential capability in dealing with situations that might be difficult for conventional nonlinear control schemes. However, to the best of our knowledge, the current literature on NN-based control lacks theoretical guarantees for stability and tracking performance. This precludes the application of NN-based control schemes to systems where stringent stability and performance guarantees are required. To address this gap, this paper proposes a systematic and comprehensive methodology to design provably-stable NN-based control schemes for affine nonlinear systems. Rigorous analysis is provided to show that the proposed approach guarantees stability of the closed-loop system with the NN in the loop. Also, it is shown that the resulting NN-based control scheme ensures that system states asymptotically converge to a neighborhood around the desired equilibrium point, with a tunable proximity threshold. The proposed methodology is validated and evaluated via simulation studies on an inverted pendulum and experimental studies on a Parrot Bebop 2 drone. 
\end{abstract}


\end{frontmatter}

\section{Introduction}\label{sec:Intro}

Nonlinear systems appear in today's real-world control problems. Historically, nonlinear systems have been addressed through various techniques of linearization and application of well-established linear control system theory. Inherently, almost all physical systems are nonlinear and the foundational linear system theory incentivized methodologies and regimes where nonlinear systems could be treated as linear systems. For instance, one can use the feedback linearization technique \cite{KhalilBook,IsidoriBook} to convert the nonlinear system into a linear system, and then use linear control techniques (see e.g., \cite{OgataBook,ChenBook}) to address the control problem. Another approach is to successively linearize the nonlinear system and use the Linear Quadratic Regulator (LQR) method to control the nonlinear system; this approach is called iterative LQR and has been widely studied in prior work, e.g.,
\cite{Li2004,Todorov2005,Rodrigues2011,Prasad2014,Boby2014,Mathiyalagan2019,Chen2017}.

A different approach to control nonlinear systems is to deal with them directly by using nonlinear control techniques; see \cite{KhalilBook,SastryBook,Rawlings1994,VidyasagarBook} for details of some of existing techniques. One of the challenges in direct controlling of nonlinear systems is guaranteeing stability at all operating points \cite{Hunt2011NeuralNE,Sinha2021AdaptiveRM}.
One possible approach to address this issue, {{{\color{blue}}}which has gained growing attention  \cite{elhaki2020robust,cheng2020neural,zhou2021control,jiang2019design,punish2018},} is to use Neural Networks (NN) in the control loop and in combination with Lyapunov theory.

% \textcolor{red}{Maybe this jumps into the lit review a little too quickly. Potentially add one or two sentences to the first paragraph saying something like: "Historically, these nonlinear systems have been addressed through various techniques of linearization and application of well-established linear system theory. Inherently, almost all physical system are nonlinear and the foundational linear system theory incentivized methodologies and regimes where nonlinear systems could be treated as linear systems. Recent research has focused more on methods of dealing with the nonlinear system directly. One of the main challenges..."}

In this context, \cite{dai2021lyapunov} proposes a NN-based controller which provides a Lyapunov function for stability purposes; however, it does not address control objectives and performance metrics. {{{\color{blue}}}Augmented neural Lyapunov control has been introduced in \cite{grande2023augmentedNLC} to address the control problem; however, \cite{grande2023augmentedNLC} does not provide theoretical analysis for stability and convergence.} A combination of LQR and an online NN has been proposed in \cite{nghi2021lqr} to stabilize an inverted pendulum in upright posture, without providing guarantees on stability. {{{\color{blue}}}In \cite{alsaade2023HHinf}, a mixed $H_2/H_{\infty}$ control has been integrated with a NN-based observer to reduce the uncertainty and to ensure the stability of unmanned aerial vehicles.} Designing a Lyapunov-based nonlinear control determined from a NN has been discussed in \cite{rego2022lyapunov}, which uses the Lyapunov theory to compute the control law; note that online computations make this method inappropriate for real-time applications. {{\color{blue}}A NN-based adaptive control has been proposed in \cite{Esfandiari2021} for affine nonlinear system, and theoretical guarantees are developed under the assumption that the NN's approximation error is very small.} A stabilizing control law for personal aerial vehicles based on exponentially stabilizing control Lyapunov functions has been developed in \cite{jang2022dnlc} and \cite{jang2023uamdyncon}, without providing a formal stability proof with the developed NN in the loop. Ref. \cite{donti2020enforcing} integrates optimization-based projection layers into a neural network-based policy to improve robustness and performance of the system; note that this method is limited to linear systems. A NN-based adaptive control has been proposed in \cite{autenrieb2019development}, which is developed based upon the nonlinear dynamic inversion approach; note that this  method cannot be applied to a wide range of systems and does not provide stability and convergence proofs. A NN-based Lyapunov-induced control law is developed in \cite{ravanbakhsh2019learner}, which determines a control Lyapunov function for the system; since this method stops as soon as one control Lyapunov function is discovered, it can lead to a poor performance.

% Synthesized simple polynomial functions, uses a demonstrator, NN learner, and verifier in an iterative process to find a valid control Lyapunov function from counterexamples. But this method also has latency for a real time system.

% \cite{nie2022deep} uses a NN for longitudinal-lateral dynamic prediction of an autonomous vehicle's dynamics, it addresses computational efficiency and stability challenges. Whereas it cannot guarantee stability of NN output.  

% \textit{In order to parameterize Lyapunov function candidates, a two-step NN optimization approach is adopted to determine the Lyapunov function and system's parameters\cite{shirin2023kernel}. However, the efficiency of the proposed NN optimization cannot handles difficult nonlinear systems. 

{{\color{blue}}Despite being promising in academic experiments, many key challenges about using NNs in control loops remain unsolved, which prevent society from deploying such approaches widely. In particular, to the best of our knowledge, the current literature on NN-based control lacks theoretical guarantees for stability and performance.} This paper aims at addressing this gap by introducing a new approach to provable NN-based control, where the controller we train with is co-developed so that it is amenable to the provable NN-based solution. First, this paper proposes a novel one-step-ahead predictive control scheme that determines the optimal control signal at any time instant, as well as a stabilizing Lyapunov function; stability and tracking properties of the proposed one-step-ahead predictive control scheme are theoretically proven by means of formal methods. This paper then discusses how to train a NN that mimics the behavior of the proposed one-step-ahead predictive control scheme. Also, it formally investigates stability and tracking properties of the closed-loop system when the trained NN is in the control loop, and characterizes the tracking error. Finally, this paper assesses the effectiveness of the proposed one-step-ahead predictive control scheme and the associated NN-based control scheme via extensive simulation studies on an inverted pendulum system and experimental results on a drone. Figure\ref{fig:paperchart} presents the general structure of the proposed approach.

\begin{figure}[!t]
    \centering
    \includegraphics[width=8.5cm]{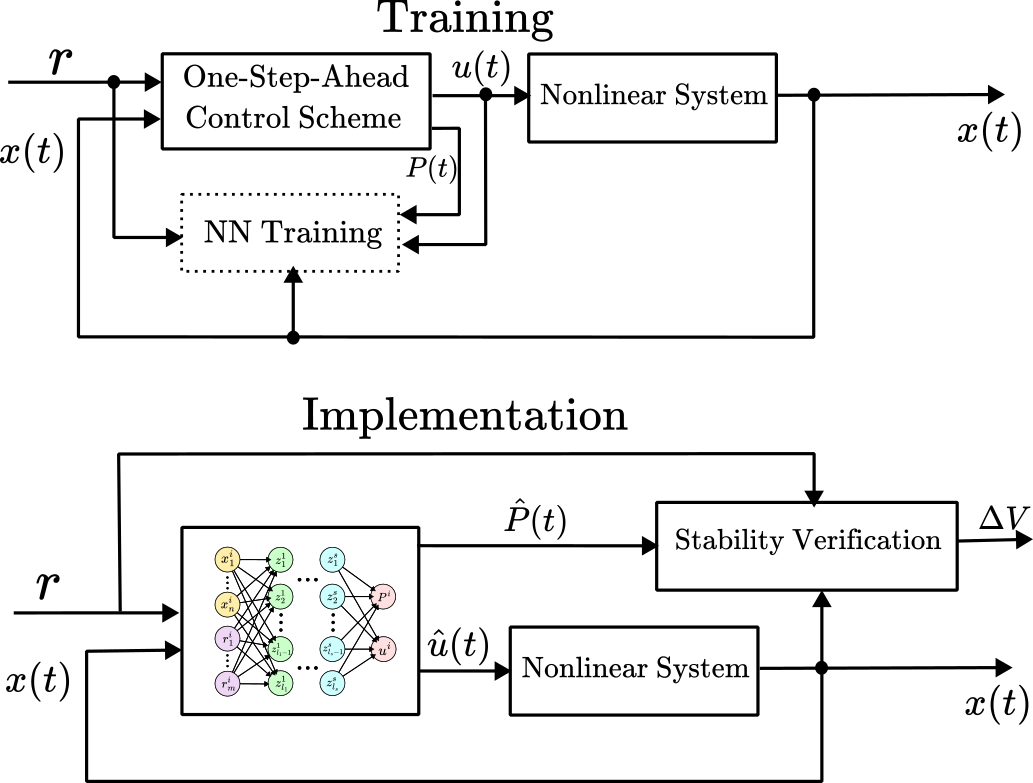}
    \caption{{{\color{blue}}General structure of proposed methodology.}}
    \label{fig:paperchart}
\end{figure}

{{\color{blue}}
The key contributions of this paper are: i) developing a novel one-step-ahead predictive control scheme for nonlinear systems and analytically proving its theoretical properties; and ii) developing a NN-based control scheme for nonlinear systems, formally proving its stability and convergence properties, and evaluating the effectiveness of the proposed control schemes via simulation and experimental studies}.

% {Note that this paper introduces a new approach to provable NN-based control, where the controller we train with is co-developed so that it is amenable to the provable NN-based solution. }

% \textcolor{red}{Might also want to point out that we are introducing a new approach to provable AI, where we co-develop the controller we train with so that it is amenable to the provable AI solution. Q: Is running the controller through the AI model any faster than computing the optimal control output. I am just wondering if we can make a claim to efficiency or scalability or some other metric in which the NN controller is "better"? It is OK if there isn't because this is kindof a foray into this notion of provable AI, but it is worth thinking about.}

The main features of the propose NN-based control scheme compared with prior work are as follows. First, rather than considering a general Lyapunov function for the system, this paper learns a stabilizing quadratic Laypunov function for each operating point, which is simpler and more straightforward; such a feature broadens the range of the control problems that can be addressed by the proposed NN-based control method. Second, this paper characterizes the degradation in the tracking performance due to learning process, and provides a single design parameter to manipulate the degradation; this feature allows the practitioners to obtain the desired tracking performance by adjusting a single parameter without concerning about the accuracy of the NN in imitating the desired control policy. Third, the developed NN-based control scheme imitates the optimal solution according to the given objective function; this feature allows the practitioners to incorporate any objective functions into the scheme without changing its structure. Fourth, this paper linearizes the nonlinear system at any time instant and determines a control signal based on the current operating point of the system; such a features allows us to apply the proposed NN-based control method (probably with some minor modification) to time-varying linear systems (we leave studying theoretical guarantees for such applications to future work). Fifth, running time of the developed NN-based control scheme is significantly smaller than that of the corresponding optimization-based scheme (our numerical analysis shows that using a NN that imitate the behavior of the one-step-ahead predictive control scheme can improve the computing efficiency by $\sim$ 319\%); we believe that the proposed approach provides useful insights for future work on how to implement sophisticated and time-consuming control schemes despite limitations on available computing power, without hampering stability or significantly degrading the performance.

% Our numerical analysis shows that even though using a NN that imitate the behavior of the one-step-ahead predictive control scheme in the loop may slightly degrade the tracking performance, it improves the computing efficiency by $\sim$ 319\%.

The rest of this paper is organized as follows. Section \ref{sec:PS} states the problem. Section \ref{sec:Solution} presents details of the proposed one-step-ahead predictive control method and discusses its theoretical properties. Section \ref{sec:Training} reports the data collection and NN training procedures. The NN-based control scheme is introduced in Section \ref{sec:NNControl} and its properties are proven analytically. Section \ref{sec:EV} evaluates the proposed method via extensive simulation studies. Finally, Section \ref{sec:Con} concludes the paper.

\paragraph*{Notation} We denote the set of real numbers by $\mathbb{R}$, the set of positive real numbers by $\mathbb{R}_{>0}$, and the set of non-negative real numbers by $\mathbb{R}_{\geq0}$. We use $\mathbb{Z}_{\geq0}$ to denote the set of non-negative integer numbers. For a matrix $A$, $A\succ0$ indicates that $A$ is positive definite, and $A\succeq0$ indicates that $A$ is positive semi-definite. We denote the transpose of matrix $A$ by $A^\top$. Given $x\in\mathbb{R}^n$ and $Q\in\mathbb{R}^{n\times n}$, $\left\Vert x\right\Vert_Q=\sqrt{\left\vert x^\top Qx\right\vert}$. Given the set $\Psi$, $|\Psi|$ indicates its cardinality. For a function $Y(x)$, $Y(x)|_{x=x^\dag}$ indicates that the function $Y(x)$ is evaluated at $x=x^\dag$. We use $I_n$ to denote $n\times n$ identity matrix.

%%%%%%%%%%%%%%%%%%%%%%%%%%%%%%%%%%%%%%%
\section{Problem Statement}\label{sec:PS}
Consider the following discrete-time affine  nonlinear system:
\begin{subequations}\label{eq:system}
\begin{align}
x(t+1) =& f\left(x(t)\right)+g\left(x(t)\right)u(t),\\
y(t)=&h\left(x(t),u(t)\right),
\end{align}
\end{subequations}
where $x(t)=[x_1(t)~...~x_n(t)] \in \mathbb{R}^n$ is the state vector at time instant $t$, $u(t)=[u_1(t)~...~u_p(t)] \in \mathbb{R}^p$ is the control input at time instant $t$, $y(t)=[y_1(t)~...~y_m(t)] \in \mathbb{R}^m$ is the output vector at time instant $t$, and $f:\mathbb{R}^n \rightarrow \mathbb{R}^n$, $g:\mathbb{R}^{n}\rightarrow \mathbb{R}^{n\times p}$, and $h:\mathbb{R}^n\times\mathbb{R}^p\rightarrow\mathbb{R}^m$ are known nonlinear functions. Let $\mathcal{X}\subseteq\mathbb{R}^n$ be the operating region\footnote{Note that this paper does not aim at enforcing the operating region as a constraint. Indeed, we will utilize the set $\mathcal{X}$ to determine the region of attraction of the proposed methods. Future work will discuss how to extend the proposed methods to guarantee state and input constraint satisfaction at all times.} of the system described in \eqref{eq:system}, i.e., $x(t)\in\mathcal{X},~\forall t\geq0$.

\begin{assumption}\label{Assumption:Linearization}
For any given $x^\dag\in\mathcal{X}$, let $A^\dag=\frac{\partial f(x)}{\partial x}|_{x=x^\dag}$ and $B^\dag=g(x^\dag)$; then, the pair $(A^\dag,B^\dag)$ is stabilizable. In other words, the linearized system around any point in the set $\mathcal{X}$ is stabilizable. 
\end{assumption}

\begin{assumption}\label{assumption:Lipchitz}
For any $x^\dag\in\mathcal{X}$, we have $\left\Vert f(x^\dag)-A^\dag x^\dag\right\Vert\leq\delta$, for some $\delta\in\mathbb{R}_{\geq0}$, where $A^\dag=\frac{\partial f(x)}{\partial x}|_{x=x^\dag}$. In other words, the linearization error can be upper-bounded with $\delta$ throughout the operating region $\mathcal{X}$. 
\end{assumption}

\begin{assumption}\label{assumption:Lipchitz2}
Functions $f(x)$ and $g(x)$ are $\mu_f$ and $\mu_g$ Lipschitz continuous, respectively, throughout the operating region $\mathcal{X}$. That is, for any $x^\dag\in\mathcal{X}$ and $x^\ddag\in\mathcal{X}$, we have $\left\Vert f(x^\dag)-f(x^\ddag)\right\Vert\leq\mu_f\left\Vert x^\dag-x^\ddag\right\Vert$ and $\left\Vert g(x^\dag)-g(x^\ddag)\right\Vert\leq\mu_g\left\Vert x^\dag-x^\ddag\right\Vert$.
\end{assumption}

Let $r \in \mathbb{R}^m$ be the desired reference. Let $\bar{x}_r$ and $\bar{u}_r$ be the steady state and control input, respectively, such that
\begin{align}\label{eq:SSconfiguration}
\bar{x}_r=f\left(\bar{x}_r\right)+g\left(\bar{x}_r\right)\bar{u}_r,~~r=h\left(\bar{x}_r,\bar{u}_r\right),
\end{align}
where $\bar{x}_r\in\mathcal{X}$. Such a reference signal is called a steady-state admissible reference; we denote the set of all steady-state admissible references by $\mathcal{R}\subseteq\mathbb{R}^m$.

This paper addresses the following problem.
\begin{problem}\label{Problem}
For any given $r\in\mathcal{R}$, design a NN-based control scheme that determines the optimal control input to steer the state of system \eqref{eq:system} to $\bar{x}_r$ and its steady input to $\bar{u}_r$. 
\end{problem}

% Suppose $r \in \mathbb{R}^n$ be the desired reference. In order to converge the system to its
% corresponding steady-state $\bar{x}_r$, every input $u(t)$ satisfies Lyapunov condition. That is, the derivative of Lyapunov function in Equation 2 should satisfy ${\Delta}V \leq 0$ for every time step.

% \begin{equation}
% \begin{matrix}
%     {\Delta}V = {(x(t+1)-\bar{x}_r)}^T{P}(x(t+1)-\bar{x}_r) \\ 
%     - {(x(t)-\bar{x}_r)}^T{P}(x(t)-\bar{x}_r) 
% \end{matrix}
% \end{equation}

% This technical note focuses on the following problem.

% \textit{problem 1}: Given $r \in \mathbb{R}^n$ as a desired reference, design a control input to steer the state of the system to the steady-state $\bar{x}_r$.

%%%%%%%%%%%%%%%%%%%%%%%%%%%%%%%%%%%%%%%%%%
\section{One-Step-Ahead Predictive Control}\label{sec:Solution}

Problem \ref{Problem} is a well-known problem in the literature and several methods have been presented to address it (see, e.g., \cite{dai2021lyapunov,jang2022dnlc,autenrieb2019development}). However, to the best of our knowledge, the current literature does not guarantee stability and convergence in the presence of training errors in an analytic manner. Also, prior work focuses on providing a stabilizing control input, without discussing optimality of the obtained solution.

To address Problem \ref{Problem}, this section proposes a novel one-step-ahead predictive control scheme that determines the optimal control input and a quadratic stabilizing Lyapunov function depending on the current states of the system. The procedure of training a NN that imitates the behavior of the developed one-step-ahead predictive control will be reported in Section \ref{sec:Training}. Finally, Section \ref{sec:NNControl} will develop a NN-based control scheme and will study its properties.

\subsection{Control Structure}\label{AA}
One possible approach to control system \eqref{eq:system} is to utilize the Model Predictive Control (MPC) framework \cite{camacho2007model,rawlings2017model}, which determines the control input by solving a receding horizon optimal control problem. While resulting problems are convex for linear systems, they are not necessarily convex for nonlinear systems \cite{Rawlings1994,Grune2011}, which creates challenges for stability proofs \cite{Allgower2000} and real-time implementations \cite{Hosseinzadeh2023RobustTermination}. One possible approach to address the above-mentioned issues is to iteratively linearize the nonlinear system and use MPC for linear systems. {{\color{blue}}It is well-known that \cite{Yin2015,HosseinzadehCANE} when the prediction horizon is large, prediction errors due to linearization (and to model uncertainty and external disturbances) can significantly degrade the performance of the system and can even lead to instability.} Prior work (e.g., \cite{Kambhampati2000}) suggests using one step ahead prediction to compute the control input at each time instant so as to bring the system states at the next time instant to a desired value. This subsection proposes a novel one-step-ahead predictive control for system \eqref{eq:system}, and investigates its stability and convergence properties. Once again, this paper does not aim at enforcing constraints on states and control inputs.

Given the desired reference $r\in\mathcal{R}$ and the current state $x(t)$, we determine the optimal control input $u^\ast(t)\in\mathbb{R}^p$ and the optimal Lyapunov matrix $P^\ast(t)\in\mathbb{R}^{n\times n}$ by solving the following optimization problem:
\begin{subequations}\label{eq:OptimizationProblemMain}{{\color{blue}}
\begin{align}\label{eq:CostFunction}
u^\ast(t),P^\ast(t)=\arg\,\min_{u,P}\,\left\Vert x^+-\bar{x}_r\right\Vert^2_{Q_x}+\left\Vert u-\bar{u}_r\right\Vert^2_{Q_u}
+V\left(x(t),r,P\right)^2,
\end{align}}
subject to the following constraints:
\begin{align}
& x^+=A_tx(t)+B_tu,\label{eq:Constraint1}\\
& P\succ0,\label{eq:Constraint2}\\
& V(x^+,r,P)-V\left(x(t),r,P\right)\leq-\theta\left\Vert x(t)-\bar{x}_r\right\Vert,\label{eq:Constraint3}
\end{align}
\end{subequations}
where $A_t$ and $B_t$ describe the linearized system around the current state $x(t)$ (see Assumption \ref{Assumption:Linearization}), $\theta\in\mathbb{R}_{>0}$ is a design parameter, $V(x,r,P):=\left\Vert x-\bar{x}_r\right\Vert_{P}=\sqrt{\left\vert(x-\bar{x}_r)^\top P(x-\bar{x}_r)\right\vert}$ is the Lyapunov function\footnote{For any $r$ and $P\succ0$, the function $V(x,r,P):=\left\Vert x-\bar{x}_r\right\Vert_{P}$ satisfies the Lyapunov conditions, i.e., i) $V(x,r,P)\geq0$ for all $x$; ii) $V(x,r,P)=0$ if and only if $x=\bar{x}_r$; and iii) $V(x,r,P)\rightarrow\infty$ if $\left\Vert x\right\Vert\rightarrow\infty$.}, $x^+$ is the one-step-ahead prediction computed based on the linearized model around $x(t)$, and $Q_x=Q_x^{\top} \succeq 0$ ($Q_x \in \mathbb{R}^{n \times n}$) and $Q_u=Q_u^{\top} \succ 0$ ($Q_u \in \mathbb{R}^{p \times p}$) are weighting matrices.

The first and second terms in the cost function \eqref{eq:CostFunction} penalizes deviation from the desired steady state $\bar{x}_r$ and steady input $\bar{u}_r$, respectively. The last term in cost function \eqref{eq:CostFunction} penalizes the Lyapunov function $V(x)$ (note that the main advantage of this term will be shown later). Constraint \eqref{eq:Constraint1} enforces the linearized dynamics, constraint \eqref{eq:Constraint2} enforces positive definiteness of the Lyapunov matrix $P$, and constraint \eqref{eq:Constraint3} indicates that the obtained control input must be stabilizing at the current time instant $t$.

\begin{remark}\label{eq:ExponentialStability}
Let system \eqref{eq:system} be linear, i.e., $x(t+1)=Ax(t)+Bu(t)$. In this case, it is well-known that one can determine a single constant Lyapunov matrix $P\succ0$. It can be shown that if the matrix $P$ satisfies $\sqrt{\lambda_{\min}(P)}>\theta$, where $\lambda_{\min}(P)$ indicates the smallest eigenvalue of $P$, constraint \eqref{eq:Constraint3} imposes exponential stability (see \cite{Aitken1994} for the definition of exponential stability in discrete-time linear systems).
\end{remark}

\begin{remark}\label{remark:LocalOptimality}
Let system \eqref{eq:system} be linear, i.e., $x(t+1)=Ax(t)+Bu(t)$. In this case, the finite-time LQR control law with prediction horizon of one to steer the states of the system to $\bar{x}_r$ can be obtained by solving the following optimization problem:
\begin{align}\label{eq:LQR}
u^\ast(t)=\arg\,\min_{u}\,\left\Vert x(t)-\bar{x}_r\right\Vert^2_{R_x}+\left\Vert u-\bar{u}_r\right\Vert^2_{R_u} +\left\Vert x(t+1)-\bar{x}_r\right\Vert^2_{R_f},
\end{align}
where $R_x=R_x^{\top} \succ 0$ ($R_x\in\mathbb{R}^{n \times n}$), $R_u=R_u^{\top} \succ 0$ ($R_u\in \mathbb{R}^{p \times p}$), and $R_f=R_f^{\top}\succeq 0$ ($R_f \in \mathbb{R}^{n \times n}$). Comparing \eqref{eq:OptimizationProblemMain} with \eqref{eq:LQR} yields that, in the case of a linear system, the proposed one-step-ahead predictive control scheme provides the finite-time LQR control law with prediction horizon of one, where the weighting matrices are $R_x=P$, $R_u=Q_u$, and $R_f=Q_x$.
\end{remark}

\begin{remark}\label{remark:InfiniteHorizonLQR}
Although this paper considers one step ahead prediction in the cost function \eqref{eq:CostFunction}, it is possible to extend the obtained results to problems with larger prediction horizon. In particular, it is possible to consider an infinite-horizon prediction horizon in \eqref{eq:CostFunction}, as $\sum_{k=1}^\infty\left\Vert \hat{x}(k)-\bar{x}_r\right\Vert^2_{Q_x}+\sum_{k=0}^\infty\left\Vert u(k)-\bar{u}_r\right\Vert^2_{Q_u}+V\left(x(t),r,P\right)^2$, where $\hat{x}(k+1)=A\hat{x}(k)+Bu(k)$ with $\hat{x}(0)=x(t)$, is the predicted state vector based on the linearized model. Such a cost function resembles the iterative LQR with infinite horizon; see, e.g., \cite{Li2004,Todorov2005}. Note that, when the linearization error $\delta$ as in Assumption \ref{assumption:Lipchitz} is large, or the system is subject to model uncertainty and/or external disturbances, considering a large prediction horizon can significantly degrade the performance, and thus it is undesired.  
\end{remark}

\subsection{Theoretical Analysis}

This section studies theoretical properties of the proposed one-step-ahead predictive control scheme. First, we study the recursive feasibility of the proposed scheme.

\begin{theorem}[Recursive Feasibility]\label{theorem:feasibility}
Consider system \eqref{eq:system} and suppose that \eqref{eq:OptimizationProblemMain} is feasible at $t=0$. Then, it remains feasible for all $t>0$.
\end{theorem}

\begin{proof}
Under Assumption \ref{Assumption:Linearization}, it is possible to determine a control input such that the dynamics of the linearized system around any point in the operating region $\mathcal{X}$ is stable. Thus, no further effort is required to show that if $x(t)\in\mathcal{X}$, the optimization problem \eqref{eq:OptimizationProblemMain} is feasible; this completes the proof. 
\end{proof}

Next, we study the stability and convergence of the one-step-ahead predictive control scheme given in \eqref{eq:OptimizationProblemMain}.

\begin{theorem}[Stability and Convergence]\label{theorem:MainController}
Consider system \eqref{eq:system} and suppose that the control scheme described in \eqref{eq:OptimizationProblemMain} is used to control it. Then, for any given $r\in\mathcal{R}$, the tracking error $\left\Vert x(t)-\bar{x}_r\right\Vert$ remains bounded, and there exists $\sigma\in\mathbb{R}_{>0}$ such that $\left\Vert x(t)-\bar{x}_r\right\Vert\leq\sigma$ as $t\rightarrow\infty$. 
\end{theorem}

\begin{proof}
Let $\big(u^\ast(t),P^\ast(t)\big)$ and  $\big(u^\ast(t+1),P^\ast(t+1)\big)$ be the optimal solutions at time instants $t$ and $t+1$, respectively. To prove this theorem, we only need to show that $\Delta V(t):=V\left(x(t+1),r,P^\ast(t+1)\right)-V\left(x(t),r,P^\ast(t)\right)<0$ whenever $\left\Vert x(t)-\bar{x}_r\right\Vert>\sigma$ for some $\sigma\in\mathbb{R}_{\geq0}$.

We have:
\begin{align}\label{eq:DV1}
\Delta V(t)=&\left\Vert x(t+1)-\bar{x}_r\right\Vert_{P^\ast(t+1)}-\left\Vert x(t)-\bar{x}_r\right\Vert_{P^\ast(t)}\nonumber\\
=&\left\Vert f\left(x(t)\right)+g\left(x(t)\right)u^\ast(t)-\bar{x}_r\right\Vert_{P^\ast(t+1)}-\left\Vert x(t)-\bar{x}_r\right\Vert_{P^\ast(t)},
\end{align}

We add and subtract the term $A_tx(t)$ from the first norm in the right-hand side of \eqref{eq:DV1}, where $A_t$ representing the linearized dynamics at time instant $t$ (see Assumption \ref{Assumption:Linearization}). Thus, according to triangle inequality\footnote{Given $z_1,z_2\in\mathbb{R}^n$ and $Q\succ0$ ($Q\in\mathbb{R}^{n\times n}$), we have $\left\Vert z_1+z_2\right\Vert_Q^2=z_1^\top Q^{\frac{1}{2}}Q^{\frac{1}{2}}z_1+z_2^\top Q^{\frac{1}{2}}Q^{\frac{1}{2}}z_2+2z_1^\top Q^{\frac{1}{2}}Q^{\frac{1}{2}}z_2\leq\left\Vert Q^{\frac{1}{2}}z_1\right\Vert^2+\left\Vert Q^{\frac{1}{2}}z_2\right\Vert^2+2\left\Vert Q^{\frac{1}{2}}z_1\right\Vert\left\Vert Q^{\frac{1}{2}}z_2\right\Vert$, which implies that $\left\Vert z_1+z_2\right\Vert_Q^2\leq\left(\left\Vert Q^{\frac{1}{2}}z_1\right\Vert+\left\Vert Q^{\frac{1}{2}}z_2\right\Vert\right)^2$, or $\left\Vert z_1+z_2\right\Vert_Q^2\leq\left(\left\Vert z_1\right\Vert_Q+\left\Vert z_2\right\Vert_Q\right)^2$. Thus, we have $\left\Vert z_1+z_2\right\Vert_Q\leq\left\Vert z_1\right\Vert_Q+\left\Vert z_2\right\Vert_Q$.}, we have:
\begin{align}\label{eq:DV2}
\Delta V(t)\leq \left\Vert f\left(x(t)\right)-A_tx(t)\right\Vert_{P^\ast(t+1)}
+\left\Vert A_tx(t)+g\left(x(t)\right)u^\ast(t)-\bar{x}_r\right\Vert_{P^\ast(t+1)}
-\left\Vert x(t)-\bar{x}_r\right\Vert_{P^\ast(t)},
\end{align}
which according to Assumption \ref{assumption:Lipchitz} implies that\footnote{Given $z\in\mathbb{R}^n$ and $Q\succ0$ ($Q\in\mathbb{R}^{n\times n}$), we have $\lambda_{\min}(Q)\left\Vert z\right\Vert^2\leq\left\Vert z\right\Vert_Q^2\leq\lambda_{\max}(Q)\left\Vert z\right\Vert^2$, which implies that $\sqrt{\lambda_{\min}(Q)}\left\Vert z\right\Vert\leq\left\Vert z\right\Vert_Q\leq\sqrt{\lambda_{\max}(Q)}\left\Vert z\right\Vert$.}:
\begin{align}\label{eq:DV2}
\Delta V(t)\leq\sqrt{\lambda_{\text{max}}\left(P^\ast(t+1)\right)}\delta-\left\Vert x(t)-\bar{x}_r\right\Vert_{P^\ast(t)}
+\left\Vert A_tx(t)+g\left(x(t)\right)u^\ast(t)-\bar{x}_r\right\Vert_{P^\ast(t+1)},
\end{align}
where $\lambda_{\text{max}}\left(P^\ast(t+1)\right)\in\mathbb{R}_{>0}$ is the maximum eigenvalue of matrix $P^\ast(t+1)$. 

Adding and subtracting the term $\left\Vert A_tx(t)+g\left(x(t)\right)u^\ast(t)-\bar{x}_r\right\Vert_{P^\ast(t)}$ to the right-hand side of the inequality \eqref{eq:DV2} yields:
\begin{align}\label{eq:DV3}
\Delta V(t)\leq&\sqrt{\lambda_{\text{max}}\left(P^\ast(t+1)\right)}\delta-\left\Vert x(t)-\bar{x}_r\right\Vert_{P^\ast(t)}
+\left\Vert A_tx(t)+g\left(x(t)\right)u^\ast(t)-\bar{x}_r\right\Vert_{P^\ast(t+1)}\nonumber\\
&-\left\Vert A_tx(t)+g\left(x(t)\right)u^\ast(t)-\bar{x}_r\right\Vert_{P^\ast(t)}
+\left\Vert A_tx(t)+g\left(x(t)\right)u^\ast(t)-\bar{x}_r\right\Vert_{P^\ast(t)},
\end{align}
which according to the fact that $\left\Vert A_tx(t)+g\left(x(t)\right)u^\ast(t)-\bar{x}_r\right\Vert_{P^\ast(t)}-\left\Vert x(t)-\bar{x}_r\right\Vert_{P^\ast(t)}\leq-\theta\left\Vert x(t)-\bar{x}_r\right\Vert$ (this is a direct implication from the optimization problem \eqref{eq:OptimizationProblemMain}) implies that:
\begin{align}\label{eq:DV4}
\Delta V(t)\leq&\sqrt{\lambda_{\text{max}}\left(P^\ast(t+1)\right)}\delta-\theta\left\Vert x(t)-\bar{x}_r\right\Vert\nonumber
+\left\Vert A_tx(t)+g\left(x(t)\right)u^\ast(t)-\bar{x}_r\right\Vert_{P^\ast(t+1)}\nonumber\\
&-\left\Vert A_tx(t)+g\left(x(t)\right)u^\ast(t)-\bar{x}_r\right\Vert_{P^\ast(t)}.
\end{align}

Let $J\big(u(t),P(t)|x(t),r\big)$ and $J\big(u(t+1),P(t+1)|x(t+1),r\big)$ be the cost of the control scheme given in \eqref{eq:OptimizationProblemMain} at time instants $t$ and $t+1$, respectively. According to the optimality of the solution $\big(u^\ast(t+1),P^\ast(t+1)\big)$ at time instant $t+1$, we have:
\begin{align}\label{eq:Optimality1}
J\big(u^\ast(t+1),P^\ast(t+1)|x(t+1),r\big)\leq 
J\big(u^\ast(t+1),P^\ast(t)|x(t+1),r\big),
\end{align}
which implies that
\begin{align}
&\left\Vert A_{t+1}x(t+1)+B_{t+1}u^\ast(t+1)-\bar{x}_r\right\Vert^2_{Q_x}+\left\Vert u^\ast(t+1)-\bar{u}_r\right\Vert^2_{Q_u}\nonumber
+\left\Vert x(t+1)-\bar{x}_r\right\Vert_{P^\ast(t+1)}^2\\
&\leq\left\Vert A_{t+1}x(t+1)+B_{t+1}u^\ast(t+1)-\bar{x}_r\right\Vert^2_{Q_x}
+\left\Vert u^\ast(t+1)-\bar{u}_r\right\Vert^2_{Q_u}+\left\Vert x(t+1)-\bar{x}_r\right\Vert_{P^\ast(t)}^2.
\end{align}

Thus, we have:
\begin{align}\label{eq:DV5}
&\left\Vert x(t+1)-\bar{x}_r\right\Vert_{P^\ast(t+1)}^2\leq\left\Vert x(t+1)-\bar{x}_r\right\Vert_{P^\ast(t)}^2.
\end{align}
or
\begin{align}\label{eq:DV5}
&\left\Vert x(t+1)-\bar{x}_r\right\Vert_{P^\ast(t+1)}\leq\left\Vert x(t+1)-\bar{x}_r\right\Vert_{P^\ast(t)}.
\end{align}

Replacing $x(t+1)$ with $f\left(x(t)\right)+g\left(x(t)\right)u^\ast(t)$ in \eqref{eq:DV5}, and adding and subtracting $A_tx(t)$, it follows from \eqref{eq:DV5} that\footnote{Given $z_1,z_2\in\mathbb{R}^n$ and $Q\succ0$ ($Q\in\mathbb{R}^{n\times n}$), following the arguments similar to footnote 3, it can be shown that $\left\Vert z_1\right\Vert_Q-\left\Vert z_2\right\Vert_Q\leq\left\Vert z_1+z_2\right\Vert_Q\leq\left\Vert z_1\right\Vert_Q+\left\Vert z_2\right\Vert_Q$.}:
\begin{align}\label{eq:DV6}
&\left\Vert A_tx(t)+g\left(x(t)\right)u^\ast(t)-\bar{x}_r\right\Vert_{P^\ast(t+1)}\leq\nonumber\\
&\left\Vert f\left(x(t)\right)-A_tx(t)\right\Vert_{P^\ast(t+1)}+\left\Vert f\left(x(t)\right)-A_tx(t)\right\Vert_{P^\ast(t)}+\left\Vert A_tx(t)+g\left(x(t)\right)u^\ast(t)-\bar{x}_r\right\Vert_{P^\ast(t)}\leq\nonumber\\
&\sqrt{\lambda_{\text{max}}\left(P^\ast(t+1)\right)}\delta+\sqrt{\lambda_{\text{max}}\left(P^\ast(t)\right)}\delta
+\left\Vert A_tx(t)+g\left(x(t)\right)u^\ast(t)-\bar{x}_r\right\Vert_{P^\ast(t)}.
\end{align}

Thus, according to \eqref{eq:DV4} and \eqref{eq:DV6}, we have:
\begin{align}\label{eq:DV7}
\Delta V(t)\leq&2\sqrt{\lambda_{\text{max}}\left(P^\ast(t+1)\right)}\delta+\sqrt{\lambda_{\text{max}}\left(P^\ast(t)\right)}\delta-\theta\left\Vert x(t)-\bar{x}_r\right\Vert.
\end{align}

Setting $\sigma=\frac{3\sqrt{\bar{\lambda}_P}\delta}{\theta}$, where $\bar{\lambda}_P=\sup_{t\geq0}\lambda_{\text{max}}\left(P^\ast(t)\right)$ ($\bar{\lambda}_P\in\mathbb{R}_{>0}$), no further effort is needed to complete the proof.
\end{proof}

\begin{remark}
Since $\theta$ is a design parameter, the upper-bound of the tracking error (i.e., $\sigma$) can be made arbitrarily small. However, our numerical experiments show that the optimizaiton problem \eqref{eq:OptimizationProblemMain} can become  numerically ill-conditioned for large values of $\theta$. Future work will provide techniques to mitigate numerical issues in the proposed methodology. 
\end{remark}

\begin{remark}
It is obvious that for linear systems, since $\delta=0$, the equilibrium point $\bar{x}_r$ is asymptotically stable with the proposed one-step-ahead predictive control given in \eqref{eq:OptimizationProblemMain}. 
\end{remark}

\begin{remark}
Theorem \ref{theorem:MainController} indicates that, despite the conventional nonlinear control methods, one can manipulate the achieved tracking error by adjusting only one scalar in the proposed one-step-ahead predictive control given in \eqref{eq:OptimizationProblemMain}. 
\end{remark}

% {\color{red}$\bullet$ Discuss data collection process in details.}

% {\color{red}$\bullet$ We need to prove recursive feasibility and closed-loop stability of the \textit{1-step prediction method} that we are using to collect data. Add term $x(t)^\top Px(t)$ to the cost function instead of $det(P)$. }

\section{Data Collection and NN Training}\label{sec:Training}
{{\color{blue}}
At any time instant $t$, the proposed one-step-ahead predictive control receives the desired reference $r$ and the current state vector $x(t)$, and solves the optimization problem \eqref{eq:OptimizationProblemMain} to compute the control input $u(t)$ and the Lyapunov matrix$P(t)$. However, solving the optimization problem \eqref{eq:OptimizationProblemMain} can be computationally challenging, and may not be realistic for real-time applications. To address this issue, we propose to train a NN to approximate the relationship between input parameters of the optimization problem \eqref{eq:OptimizationProblemMain} (i.e., state vector $x(t)$ and desired reference $r$) and output parameters of the optimization problem \eqref{eq:OptimizationProblemMain} (i.e., control input $u(t)$ and Lyapunov matrix $P(t)$), and use it in the loop to control system \eqref{eq:system}. Our intuition is that using the NN significantly decreases the computational burden of the proposed control scheme, as it reduces the problem of computing $u(t)$ and $P(t)$ into simple function evaluations. To do so, this section discusses the data collection and training procedure to train a NN that imitates the behavior of the one-step-ahead predictive control scheme developed in Section \ref{sec:Solution} in the operating region $\mathcal{X}$. Section \ref{sec:NNControl} will investigate the theoretical properties of the closed-loop system with the NN in the loop in the presence of NN's approximation errors.
}

To obtain the training dataset, first, we divide the operating region $\mathcal{X}$ and the set of steady-state admissible references $\mathcal{R}$ into $N_x$ and $N_r$ grid cells, respectively; this process generates $N_x\cdot N_r$ data points which are denoted by $(x^i,r^i),~i=1,\cdots,N_x\cdot N_r$, where $x^i\in\mathcal{X}$ and $r^i\in\mathcal{R}$. Then, for each data point $(x^i,r^i)$, we solve the optimization problem \eqref{eq:OptimizationProblemMain} under the assumption that $x(t)=x^i$; this yields the optimal control input (denoted by $u^i$) and optimal Lyapunov matrix (denoted by $P^i$) for the data point $(x^i,r^i)$. Finally, the training dataset can be constructed as follows:
\begin{align}
\mathbb{D}_{training}=\left\{\left(x^i,r^i,u^i,P^i\right),~i=1\cdots,N_x\cdot N_r\right\},
\end{align}
where the tuple $\left(x^i,r^i,u^i,P^i\right)$ is the $i$th training data point. It is obvious that $|\mathbb{D}_{training}|=N_x\cdot N_r$. {{\color{blue}}Note that one can use a larger $N_x$ and $N_r$ (i.e., divide the sets $\mathcal{X}$ and $\mathcal{R}$ into smaller cells) to make sure that enough training data is collected and the training dataset $\mathbb{D}_{training}$ covers all relevant aspects of the problem domain. Thus, the NN can be trained offline and there is no need for online adjustments to the network's parameters.}

% To collect the valid training dataset that satisfies \ref{eq:OptimizationProblemMain} and its constrains, given a arbitrary state $x_0\in\mathcal{X}$ and reference $r\in\mathcal{R}$, we use a general optimization algorithms to acquire the optimal solution $\big(u^\ast(t),P^\ast(t)\big)$ for each location inside a region $\mathcal{X}$ with interval ${\Delta}x$. For instance, the $i_{th}$ data has input state $(x_1,x_2,...,x_n)_i$ and it corresponding label $(u^\ast,P^\ast)_i$. 
% \begin{equation}\label{eq:lya}
%     V =  {(x(t)-\bar{x}_r)}^\top{P}(x(t)-\bar{x}_r) \geq 0
% \end{equation}

% The optimization toolbox in MATLAB is applied to find optimized input $u$ and its related $P$ matrix. To be more specific, we set a objective function:

% \begin{equation}
%      F = 0.1||u||^2+ ||x(t+1)-\bar{x}_r||^2 +0.1x(t)^\top Px(t)
% \end{equation}

% In order to satisfy Lyapunov condition, we add constrains:
% \begin{enumerate}
%     \item $x(t)^\top Px(t)\geq0.1$;
%     \item $P$ is positive definite;
%     \item ${\Delta}V=V_{t+1}-V_t<-{\theta}||x||$.
% \end{enumerate}

% Notice the final term $\theta=0.01$. In a $\mathbb{R}^{2}$ scenario, $P$ is a $2\times2$ symmetric matrix. we generate data set $X=\{[x_1,x_2]_k\}_{k=1}^m$, where $-5{\leq}x_1{\leq}5$,$-5{\leq}x_2{\leq}5$, both of $x_1$ and $x_2$ have 0.1 increment interval. 

Once the training dataset $\mathbb{D}_{training}$ is collected, we use them to train a (deep) feedforward NN \cite{Fine1999}. {{\color{blue}}Feedforward NNs are flexible and scalable, and can be used to approximate complex relationships between input and out parameters; also, their training is relatively straightforward compared to other structures \cite{Goodfellow2016}}. Figure\ref{fig:nn} presents the general structure of the considered NN, where $x^i$ and $r^i$ are the inputs of the NN and $u^i$ and $P^i$ are its outputs. {{\color{blue}}Note that although we do not claim to be optimal, our experiments suggest that a feedforward NN provides satisfactory performance.}  {{\color{blue}}We use the growing method (see, e.g., \cite{Evci2022,Wu2020,Bengio2005}) to determine the number of hidden layers and neurons; that is, we start with a small network and increase the complexity until the desired accuracy is achieved.} Note that we apply the dropout method \cite{srivastava2014dropout,lim2021study,sanjar2020weight,wu2015towards} to avoid overfitting. Also, to improve the training process, we use the cosine annealing strategy \cite{liu2022super,eshraghian2022navigating} to adjust the learning rate.  {{\color{blue}}To evaluate the network's convergence, we use the mean squared error (MSE) loss function; see more details about the role of loss function in, e.g., \cite{Goodfellow2016}.}

% {\color{red} Every neurons in the same layer are fully connected with neurons from previous layer and neurons from next layer. The structure is simple and commonly applied on various regression tasks.} {\color{red} Other common NN structures such as long-short-term-memory(LSTM) structure and convolutional neural network (CNN) were failed to convergent in our experiments and therefore we currently only consider feedforward NN structure.} 

\begin{figure}[!h]
    \centering
    \includegraphics[width=6cm]{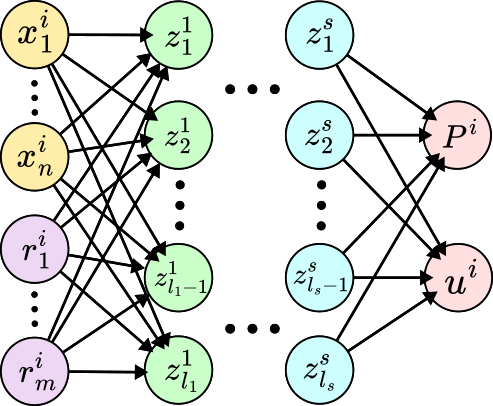}
    \caption{{{\color{blue}}The NN trained to approximate the solution of \eqref{eq:OptimizationProblemMain}, with $s$ hidden layers and $l_i$ neurons in the $i$th hidden layer.}}
    \label{fig:nn}
\end{figure}

% And a changing learning rate is set with cosine annealing to improve the training process. Besides, drop out method is applied to avoid model overfitting. Training results are evaluated by total loss and testing loss. 
% To balance samples, model output weights are considered.

% For tuning other hyperparameters (such as learning rate, activation function, etc), we exploit the existing well-stablished literature (see, e.g., \cite{Hagan1996,Goodfellow2016,Chollet2021}). 

% errors, loss from $P$ matrix is multiplied  10 times while loss from $u(t)$ is multiplied 100 times to balance weights of both outputs. 

% {\color{red}$\bullet$ Training process in details; considered type of the NN; other details.}

\section{NN-Based Control}\label{sec:NNControl}
Let a NN be trained as discussed in Section \ref{sec:Training} to imitate the behavior of the one-step-ahead predictive control scheme described in Section \ref{sec:Solution}. At each time instant $t$, let $\hat{u}(t)=u^\ast(t)+\Delta u(t)$ and $\hat{P}(t)=P^\ast(t)+\Delta P(t)$ be the output of the NN, where $\Delta u(t)$ and $\Delta P(t)$ indicate the difference between the optimal solutions and the NN outputs. Note that $\hat{P}(t)$ is not necessarily positive definite. We reasonably assume that \cite{dai2021lyapunov} the NN functions satisfy the bounded condition for all states in the operating region $\mathcal{X}$; this implies that $\Delta u(t)$ and $\Delta P(t)$ are bounded at any time instant $t$. We let $\bar{\Delta u}=\sup_{t\geq0}\left\Vert\Delta u(t)\right\Vert$ and $\bar{\Delta P}=\sup_{t\geq0}\left\Vert\Delta P(t)\right\Vert$, where $\bar{\Delta u},\bar{\Delta P}\in\mathbb{R}_{\geq0}$.

\subsection{Theoretical Analysis}

{{\color{blue}}The following theorem shows that if the trained NN is utilized in the loop to control system \eqref{eq:system}, the system is stable and the tracking error remains bounded.} See Figure\ref{fig:paperchart} for the general structure of the proposed NN-based control scheme.

\begin{theorem}\label{theorem:NNControl}
{{\color{blue}}Consider system \eqref{eq:system}, and suppose that a NN trained to imitate the behavior of the control scheme \eqref{eq:OptimizationProblemMain} is utilized in the control loop. Let $r\in\mathcal{R}$ be the desired reference signal. Then, for sufficiently large $\theta$, the tracking error $\left\Vert x(t)-\bar{x}_r\right\Vert$ remains bounded, and there exists $\vartheta\in\mathbb{R}_{>0}$ such that $\left\Vert x(t)-\bar{x}_r\right\Vert\leq\vartheta$ as $t\rightarrow\infty$. }
\end{theorem}

\begin{proof}
Let $\Delta V(t):=V\left(x(t+1),r,\hat{P}(t+1)\right)-V\left(x(t),r,\hat{P}(t)\right)$. We have:{{\color{blue}}
\begin{align}\label{eq:DVNN1}
\Delta V(t)&=\left\Vert x(t+1)-\bar{x}_r\right\Vert_{P^\ast(t+1)+\Delta P(t+1)}-\left\Vert x(t)-\bar{x}_r\right\Vert_{P^\ast(t)+\Delta P(t)}\nonumber\\
&=\left\Vert f\left(x(t)\right)+g\left(x(t)\right)\left(u^\ast(t)+\Delta u(t)\right)-\bar{x}_r\right\Vert_{P^\ast(t+1)+\Delta{P(t+1)}}-\left\Vert x(t)-\bar{x}_r\right\Vert_{P^\ast(t)+\Delta P(t)}.
\end{align}
}

According to triangle inequality (see footnote 3), Equation \eqref{eq:DVNN1} can be rewritten as:{{\color{blue}}
\begin{align}\label{eq:DVNN2}
\Delta V(t)\leq&\left\Vert f\left(x(t)\right)+g\left(x(t)\right)u^\ast(t)-\bar{x}_r\right\Vert_{P^\ast(t+1)+\Delta{P(t+1)}}+\left\Vert g\left(x(t)\right)\Delta{u(t)}\right\Vert_{P^\ast(t+1)+\Delta{P(t+1)}}\nonumber\\
&-\left\Vert x(t)-\bar{x}_r\right\Vert_{P^\ast(t)+\Delta P(t)},
\end{align}
}

which implies that\footnote{Given $z\in\mathbb{R}^n$ and $Q_1,Q_2\in\mathbb{R}^{n\times n}$, let $\left\Vert z\right\Vert_{Q_1}=\sqrt{\left\vert z^\top Q_1z\right\vert}$, $\left\Vert z\right\Vert_{Q_2}=\sqrt{\left\vert z^\top Q_2z\right\vert}$, and $\left\Vert z\right\Vert_{Q_1+Q_2}=\sqrt{\left\vert z^\top (Q_1+Q_2)z\right\vert}$. Thus, it can be easily shown that $\left\Vert z\right\Vert_{Q_1}-\left\Vert z\right\Vert_{Q_2}\leq\left\Vert z\right\Vert_{Q_1+Q_2}\leq\left\Vert z\right\Vert_{Q_1}+\left\Vert z\right\Vert_{Q_2}$. Also, $\left\Vert z\right\Vert_{Q_1}\leq\sqrt{\left\Vert Q_1\right\Vert}\left\Vert z\right\Vert$ and $\left\Vert z\right\Vert_{Q_2}\leq\sqrt{\left\Vert Q_2\right\Vert}\left\Vert z\right\Vert$.}:{{\color{blue}}
\begin{align}\label{eq:DVNN3}
\Delta V(t)\leq&\left\Vert f\left(x(t)\right)+g\left(x(t)\right)u^\ast(t)-\bar{x}_r\right\Vert_{P^\ast(t+1)}+\sqrt{\bar{\Delta{P}}}\left\Vert f\left(x(t)\right)+g\left(x(t)\right)u^\ast(t)-\bar{x}_r\right\Vert\nonumber\\
&+\left\Vert g\left(x(t)\right)\Delta{u(t)}\right\Vert_{P^\ast(t+1)+\Delta{P(t+1)}}-\left\Vert x(t)-\bar{x}_r\right\Vert_{P^\ast(t)}+\sqrt{\bar{\Delta{P}}}\left\Vert x(t)-\bar{x}_r\right\Vert.
\end{align}
}

% \begin{figure}[!t]
%     \centering
%     \includegraphics[width=0.47\textwidth]{images/flowchart.png}
%     \caption{General structure of the proposed provably-stable NN based control.}
%     \label{fig:flowchart}
% \end{figure}

According to Theorem \ref{theorem:MainController}, it follows from \eqref{eq:DVNN3} that:
\begin{align}\label{eq:DVNN4}
\Delta V(t)\leq&{3\sqrt{\bar{\lambda}_P}\delta}-\theta\left\Vert x(t)-\bar{x}_r\right\Vert+\sqrt{\bar{\Delta{P}}}\left\Vert x(t)-\bar{x}_r\right\Vert+\left\Vert g\left(x(t)\right)\Delta{u(t)}\right\Vert_{P^\ast(t+1)+\Delta{P(t+1)}}\nonumber\\
&+\sqrt{\bar{\Delta{P}}}\left\Vert f\left(x(t)\right)+g\left(x(t)\right)u^\ast(t)-\bar{x}_r\right\Vert.
\end{align}

% Adding and subtracting the term $A_tx(t)+g\left(x(t)\right)u^\ast(t)$ from the first and seconds terms of \eqref{eq:DVNN1} and following the same procedures described in the proof of Theorem \ref{theorem:MainController}, it follows from \eqref{eq:DVNN1} that:
% \begin{align}
% \Delta V(t)\leq&\left(3\bar{\lambda}_P+\left\Vert\Delta P(t+1)\right\Vert\right)\delta^2-\theta\left\Vert x(t)-\bar{x}_r\right\Vert^2\nonumber\\
% &+\left\Vert g\left(x(t)\right)\Delta u(t)\right\Vert^2_{P^\ast(t+1)}\nonumber\\
% &+\left\Vert g\left(x(t)\right)\Delta u(t)\right\Vert^2_{\Delta P(t+1)}\nonumber\\
% &-\left\Vert x(t)-\bar{x}_r\right\Vert^2_{\Delta P(t)}\nonumber\\
% &+\left\Vert A_tx(t)+g\left(x(t)\right)u^\ast(t)-\bar{x}_r\right\Vert^2_{\Delta P(t+1)},
% \end{align}
% which implies that:
% \begin{align}\label{eq:DVNN2}
% \Delta V(t)\leq&\left(3\bar{\lambda}_P+\bar{\Delta P}\right)\delta^2-\theta\left\Vert x(t)-\bar{x}_r\right\Vert^2\nonumber\\
% &+\bar{\Delta P}\left\Vert x(t)-\bar{x}_r\right\Vert^2+\left\Vert g\left(x(t)\right)\Delta u(t)\right\Vert^2_{P^\ast(t+1)}\nonumber\\
% &+\left\Vert g\left(x(t)\right)\Delta u(t)\right\Vert^2_{\Delta P(t+1)}\nonumber\\
% &+\left\Vert A_tx(t)+g\left(x(t)\right)u^\ast(t)-\bar{x}_r\right\Vert^2_{\Delta P(t+1)},
% \end{align}
% where $\bar{\lambda}_P$ and $\bar{\Delta P}$ are defined above. 

At this stage, we upper bound the fourth and fifth terms in \eqref{eq:DVNN4} as follows:
\begin{itemize}
\item Fourth Term: This term can be upper bounded as (see footnote 6):
\begin{align}
\left\Vert g\left(x(t)\right)\Delta{u(t)}\right\Vert_{P^\ast(t+1)+\Delta{P(t+1)}}\leq&\left\Vert g\left(x(t)\right)\Delta{u(t)}\right\Vert_{P^\ast(t+1)}+\left\Vert g\left(x(t)\right)\Delta{u(t)}\right\Vert_{\Delta{P(t+1)}}\nonumber\\
\leq&\left(\sqrt{\bar{\lambda_p}}+\sqrt{\bar{\Delta{P}}}\right)\left\Vert g\left(x(t)\right)\Delta{u(t)}\right\Vert.\label{eq:UpperBound0}
\end{align}
    
By adding and subtracting $g\left(\bar{x}_r\right)\Delta{u(t)}$ in the norm of right-hand side, and using the Cauchy–Schwarz inequality and triangle inequality, it follows from \eqref{eq:UpperBound0} that:
\begin{align}
 \left\Vert g\left(x(t)\right)\Delta{u(t)}\right\Vert_{P^\ast(t+1)+\Delta{P(t+1)}}\leq&(\sqrt{\bar{\lambda_p}}+\sqrt{\bar{\Delta{P}}})\left\Vert g\left(x(t)\right)-g(\bar{x}_r)\right\Vert\left\Vert\Delta{u(t)}\right\Vert\nonumber\\
&+(\sqrt{\bar{\lambda_p}}+\sqrt{\bar{\Delta{P}}})\left\Vert g(\bar{x}_r)\right\Vert\left\Vert\Delta{u(t)}\right\Vert.\label{eq:UpperBound1}
\end{align}

Finally, according to Assumption \ref{assumption:Lipchitz2}, inequality \ref{eq:UpperBound1} can be expressed as:
\begin{align}
& \left\Vert g\left(x(t)\right)\Delta{u(t)}\right\Vert_{P^\ast(t+1)+\Delta{P(t+1)}}\leq(\sqrt{\bar{\lambda_p}}+\sqrt{\bar{\Delta{P}}})\mu_g\Delta\bar{u}\left\Vert x(t)-\bar{x}_r\right\Vert+(\sqrt{\bar{\lambda_p}}+\sqrt{\bar{\Delta{P}}})\left\Vert g(\bar{x}_r)\right\Vert\Delta\bar{u}.\label{eq:UpperBound2}
\end{align}

Note that for any given $r\in\mathcal{R}$, $\left\Vert g(\bar{x}_r)\right\Vert$ is bounded.

\item Fifth Term: By adding and subtracting $A_tx(t)$ and according to Assumption \ref{assumption:Lipchitz} and triangle inequality, we have:
\begin{align}
&\sqrt{\bar{\Delta{P}}}\left\Vert f\left(x(t)\right)+g\left(x(t)\right)u^\ast(t)-\bar{x}_r\right\Vert\leq\sqrt{\bar{\Delta{P}}}\delta+\sqrt{\bar{\Delta{P}}}\left\Vert A_tx(t)+g\left(x(t)\right)u^\ast(t)-\bar{x}_r\right\Vert.
\end{align}

At any time instant $t$, optimization problem \eqref{eq:OptimizationProblemMain} ensures that $\left\Vert A_tx(t)+g\left(x(t)\right)u^\ast(t)-\bar{x}_r\right\Vert_{P^{\ast}(t)}\leq\left\Vert x(t)-\bar{x}_r\right\Vert_{P^{\ast}(t)}$. Thus, according to footnote 4, we have:
\begin{align}               &\sqrt{\underline{\lambda}_P}\left\Vert A_tx(t)+g\left(x(t)\right)u^\ast(t)-\bar{x}_r\right\Vert\leq\sqrt{\bar{\lambda}_P}\left\Vert x(t)-\bar{x}_r\right\Vert,\label{eq:UpperBound3}
\end{align} 
where $\underline{\lambda}_P:=\inf_{t\geq0}\lambda_{\text{min}}\left(P^\ast(t)\right)$ ($\underline{\lambda}_P\in\mathbb{R}_{>0}$), with  $\lambda_{\text{min}}\left(P^\ast(t)\right)$ being the smallest eigenvalue of matrix $P^\ast(t)$, and $\bar{\lambda}_P$ is defined in the proof of Theorem \ref{theorem:MainController}. Thus, according to \eqref{eq:UpperBound3}, it follows from \eqref{eq:UpperBound2} that:
\begin{align}
&\sqrt{\bar{\Delta{P}}}\left\Vert f\left(x(t)\right)+g\left(x(t)\right)u^\ast(t)-\bar{x}_r\right\Vert\leq\sqrt{\bar{\Delta{P}}}\delta+\sqrt{\bar{\Delta{P}}}\frac{\sqrt{\bar{\lambda_P}}}{\sqrt{\underline{\lambda}_P}}\left\Vert x(t)-\bar{x}_r\right\Vert.\label{eq:UpperBound4}
\end{align}
\end{itemize}

Finally, combining \eqref{eq:DVNN4}, \eqref{eq:UpperBound2}, and \eqref{eq:UpperBound4} yields:
\begin{align}               
\Delta{V(t)}\leq&{3\sqrt{\bar{\lambda}_P}\delta}+\sqrt{\bar{\Delta{P}}}\delta+(\sqrt{\bar{\lambda_p}}+\sqrt{\bar{\Delta{P}}})\left\Vert g(\bar{x}_r)\right\Vert\Delta\bar{u}\nonumber\\
&-\Bigg(\theta-\sqrt{\bar{\Delta{P}}}-\frac{\sqrt{\bar{\Delta{P}}}\sqrt{\bar{\lambda_P}}}{\sqrt{\underline{\lambda}_P}}-(\sqrt{\bar{\lambda_p}}+\sqrt{\bar{\Delta{P}}})\mu_g\Delta\bar{u}\Bigg)\left\Vert x(t)-\bar{x}_r\right\Vert,
\end{align}
which completes the proof by selecting $\theta>\sqrt{\bar{\Delta{P}}}+\frac{\sqrt{\bar{\Delta{P}}}\sqrt{\bar{\lambda_P}}}{\sqrt{\underline{\lambda}_P}}+(\sqrt{\bar{\lambda_p}}+\sqrt{\bar{\Delta{P}}})\mu_g\Delta\bar{u}$, and setting $\vartheta$ as follows:
\begin{align}\label{eq:vartheta}
\vartheta=\frac{{3\sqrt{\bar{\lambda}_P}\delta}+\sqrt{\bar{\Delta{P}}}\delta+(\sqrt{\bar{\lambda_p}}+\sqrt{\bar{\Delta{P}}})\left\Vert g(\bar{x}_r)\right\Vert\Delta\bar{u}}{\theta-\sqrt{\bar{\Delta{P}}}-\frac{\sqrt{\bar{\Delta{P}}}\sqrt{\bar{\lambda_P}}}{\sqrt{\underline{\lambda}_P}}-(\sqrt{\bar{\lambda_p}}+\sqrt{\bar{\Delta{P}}})\mu_g\Delta\bar{u}}.
\end{align}
\end{proof}

% which implies that:
% \begin{align}\label{eq:DVNN2}
% &\Delta V(t)\leq\lambda_{\text{max}}\left(P^\ast(t+1)\right)\delta^2-\theta\left\Vert x(t)-\bar{x}_r\right\Vert^2\nonumber\\
% &+\left\Vert g\left(x(t)\right)\Delta u(t)\right\Vert^2_{P^\ast(t+1)}+\left\Vert g\left(x(t)\right)\Delta u(t)\right\Vert^2_{\Delta P(t+1)}\nonumber\\
% &-\left\Vert x(t)-\bar{x}_r\right\Vert^2_{\Delta P(t)}\nonumber\\
% &+\left\Vert f\left(x(t)\right)+g\left(x(t)\right)u^\ast(t)-\bar{x}_r\right\Vert^2_{\Delta P(t+1)}.
% \end{align}

% \begin{align}
% &\left\Vert f\left(x(t)\right)+g\left(x(t)\right)\left(u^\ast(t)+\Delta u(t)\right)-\bar{x}_r\right\Vert^2_{P^\ast(t+1)}\leq\nonumber\\
% &\left\Vert f\left(x(t)\right)-A_tx(t)\right\Vert^2_{P^\ast(t+1)}+\left\Vert g\left(x(t)\right)\Delta u(t)\right\Vert^2_{P^\ast(t+1)}\nonumber\\
% &+\left\Vert A_tx(t)+g\left(x(t)\right)u^\ast(t)-\bar{x}_r\right\Vert^2_{P^\ast(t+1)}
% \end{align}

\begin{remark}
{{\color{blue}}Since $\theta$ is a design parameter, the upper-bound of the tracking error when a NN is in the control loop (i.e., $\vartheta$ given in \eqref{eq:vartheta}) can be made arbitrarily small.}
\end{remark}

\begin{remark}
According to \eqref{eq:vartheta}, in the presence of an \textit{ideal} NN (i.e., $\bar{\Delta u}=\bar{\Delta P}=0$), the properties of the main controller (i.e., properties discussed in Theorem \ref{theorem:MainController}) are recovered, i.e., $\vartheta=\sigma$. 
\end{remark}

% For all the points in our dataset we acquired, they are satisfied:
% \begin{equation}
%     {\Delta}V \leq -\varepsilon||x||^2
% \end{equation}
% For training process, we have verified the training error $E$
% \begin{equation}
%     Inf\{E\}={\delta}< \infty
% \end{equation}

% We can find:

% \begin{equation}
%     {\Delta}V = {\Delta}V_{ideal} + {\delta}
% \end{equation}

% Which is equals to:

% \begin{equation}
%     {\Delta}V \leq -\varepsilon||x||^2 + {\delta}
% \end{equation}

% The euqation proofs the smoothness of our NN model output, and the output is always bounded for a unknown input points that didn't appear in our dataset.

\subsection{Region of Attraction}\label{sec:RoA}
The proposed one-step-ahead predictive control given in \eqref{eq:OptimizationProblemMain} and the NN-based control described in Section \ref{sec:Training} are developed such that the stability of system \eqref{eq:system} is guaranteed in the operating region $\mathcal{X}$. However, they do not guarantee that the trajectory of the system remains in the operating region $\mathcal{X}$ at all times. Thus, it is important to determine the Region of Attraction (RoA) of the proposed methods.

For any steady-state admissible reference $r$, the RoA is defined as the set of all initial conditions belonging to the set $\mathcal{X}$ such that ensued trajectory remains inside the set $\mathcal{X}$. In mathematical terms, for a given $r\in\mathcal{R}$, the RoA $\Phi(r)\subseteq\mathcal{X}$ including $\bar{x}_r$ can be defined as:
\begin{align}\label{eq:RoA}
\Phi(r)=\left\{x\in\mathcal{X}|\hat{x}(k|x,r)\in\mathcal{X},~k\in\mathbb{Z}_{\geq0}\right\}, 
\end{align}
where $\hat{x}(k|x,r)$ is the predicted state from the initial state $x$ at the prediction instant $k$, when the desired reference is $r$, and the one-step-ahead predictive scheme \eqref{eq:OptimizationProblemMain} or the NN-based scheme described in Section \ref{sec:NNControl} are employed to control system \ref{eq:system}.

The most intuitive way to estimate the set $\Phi(r)$ is described below. First, divide the set $\mathcal{R}$ into $n_{\mathcal{R}}$ non-overlapping regions $\mathcal{R}_i,~i\in\{1,\cdots,n_{\mathcal{R}}\}$; this can be done by using existing techniques, e.g., Delaunay tessellation \cite{Watson1981,Lawson1986}. Then, for each region $\mathcal{R}_i$, start with a large set $\Phi(r)$ (usually the set $\mathcal{X}$) including the set of steady-state admissible equilibria associated with the region $\mathcal{R}_i$ (i.e., $\bar{\mathcal{R}}_i=\left\{x|\exists r\in\mathcal{R}_i\text{ such that }x=\bar{x}_r\right\}$), and compute simulated trajectories to find a $\tilde{x}\in\mathcal{X}$ and $\tilde{r}\in\mathcal{R}_i$ such that $\hat{x}(t|\tilde{x},\tilde{r})\not\in\mathcal{X}$ at some $t$; see Figure\ref{fig:RoA} for a geometric illustration. If such a point is found, the set $\Phi(r)$ is falsified and should be shrunk by excluding a neighbourhood around $\tilde{x}$. This falsification procedure should be computed until all simulated trajectories remain in $\mathcal{X}$ at all times. Note that although the above-mentioned procedure can be expensive from the computational viewpoint, it should be performed only at design-time and does not require human intervention.

\begin{figure}[!h]
    \centering
    \includegraphics[width=6cm]{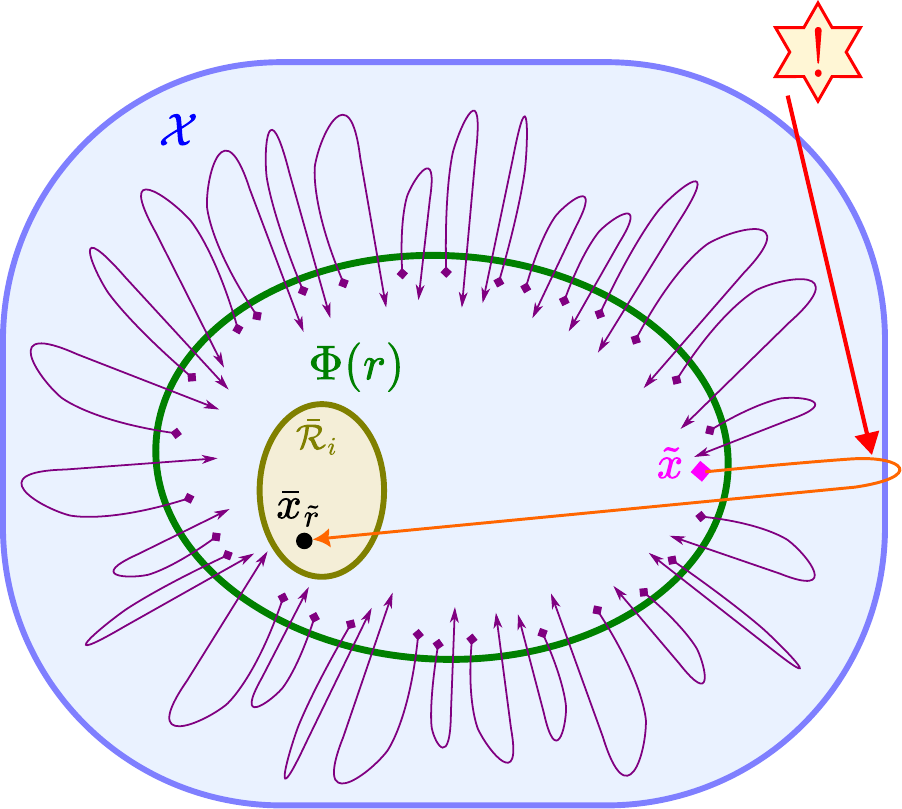}
    \caption{{{\color{blue}}Geometric illustration of the falsification procedure for determining the RoA $\Phi(r)$.}}
    \label{fig:RoA}
\end{figure}

\section{Simulation Results}\label{sec:EV}
This section assesses the effectiveness of the proposed methodology on an inverted pendulum system; see Figure\ref{fig:InvertedPendulum}. {{\color{blue}}The dynamics of the system can be described as \cite{boubaker2013IP}}:
\begin{align}\label{eq:systemInverted}
mL^2\ddot{\alpha} = mgL\sin{\alpha}+\tau,
\end{align}
where $\alpha$ represents angle of the inverted pendulum, $\tau$ is the torque generated from a motor rotating the pendulum, and $m=1$ [kg] and $L=1$ [m] are the pendulum mass and the distance from the center of mass, respectively. We assume that $\mathcal{X}=[-5,5]\times[-5,5]$ and $\mathcal{R}=[-1,1]$. The state-space representation for this system is:
\begin{subequations}
\begin{align}
\dot{x}_1&=x_2,\\
\dot{x}_2&=g\sin\left(x_1\right)+\tau,
\end{align}
\end{subequations}
where $x_1=\alpha$ and $x_2=\dot{\alpha}$. We use Euler's method to discretize system \eqref{eq:systemInverted} with the sampling period of $\Delta T=0.1$ seconds, i.e.,
\begin{subequations}\label{eq:DisPendulum}
\begin{align}
x_1(t+1)&=x_1(t)+\Delta T\cdot x_2(t),\\
x_2(t+1)&=x_2(t)+\Delta T\cdot g\cdot\sin\left(x_1(t)\right)+\Delta T\cdot\tau(t),
\end{align}
\end{subequations}

{{\color{blue}}We use the proposed NN-based scheme to control system \eqref{eq:DisPendulum}. First, we grid the above-mentioned sets $\mathcal{X}$ and $\mathcal{R}$ with steps of 0.1 to generate 200,000 data points. Then, we use YALMIP \cite{Lofberg2004} to solve the optimization problem \eqref{eq:OptimizationProblemMain} with $Q_x=2I_2$ and $Q_u=0.1$ for each data point. Once the training dataset $\mathbb{D}_{training}$ is generated, we use \texttt{Pytorch} package \cite{Pytorch} and \texttt{Adam} optimizer \cite{Adam} to train a feedforward NN with 6 hidden layers, and with 8, 32, 64, 64, 32, and 16 neurons in the hidden layers; note that the training converges within 10000 epochs given a learning rate of 0.001. Finally, we implement the resulting NN-based control scheme on \texttt{Python 3.10}.}

% We assume that $\mathcal{X}=[-5,5]\times[-5,5]$ and $\mathcal{R}=[-1,1]$, and generate the training dataset $\mathbb{D}_{training}$ by girding the above-mentioned regions with steps of 0.1; that is, $|\mathbb{D}_{training}|=200,000$. We use YALMIP toolbox \cite{Lofberg2004} to solve the optimization problem \eqref{eq:OptimizationProblemMain} for each data point, {where $Q_x=2I_2$ and $Q_u=0.1$}. After collecting data, we trained a feedforward NN with 6 hidden layers; we considered 8, 32, 64, 64, 32, and 16 neurons in the hidden layers, respectively. We use \texttt{Pytorch} package \cite{Pytorch} and \texttt{Adam} optimizer \cite{Adam} to train and test NN models for different values of $\theta$; all model converged within 10000 epochs given a learning rate 0.001. We implement the resulting control schemes on \texttt{Python 3.10}. 

\subsection{Determining the RoA}
{{\color{blue}}We use the method described in Subsection \ref{sec:RoA} to determine the RoA for $r=0$ and with different values of $\theta$. More precisely, we divided the operating region $\mathcal{X}$ into $10^6$ grid cells, and used the developed NN-based control to compute simulated trajectory with the initial condition being equal to each cell; the initial condition belongs to the RoA if the trajectory remains entirely inside the operating region $\mathcal{X}$.} The obtained RoA with $\theta=0.01$ is shown in Figure\ref{fig:ROAPendulum}. Note that the obtained RoA $\Phi(0)$ with $\theta=0.001$ and $\theta=0.0001$ are almost the same as for $\theta=0.01$, and thus are not presented.

\begin{figure}[!t]
    \centering
    \includegraphics[width=2cm]{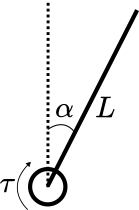}
    \caption{{{\color{blue}}Inverted pendulum system.}}
    \label{fig:InvertedPendulum}
\end{figure}

\begin{figure}[!t]
    \centering
    \includegraphics[width=7cm]{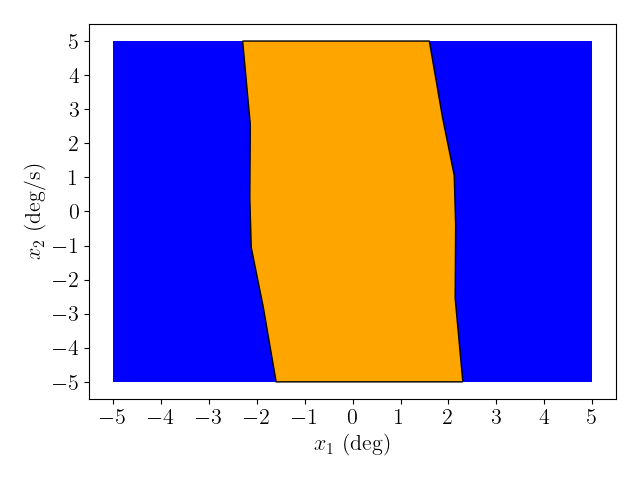}
    \caption{{{\color{blue}}The obtained RoA with $\theta=0.01$ and $r=0$; the blue and yellow represent $\mathcal{X}$ and $\Phi(0)$, respectively.}}
    \label{fig:ROAPendulum}
\end{figure}

\subsection{Performance Analysis}\label{sec:ExtensiveSimulation}
To evaluate the effectiveness of the proposed one-step-ahead predictive control scheme and the NN-based control scheme, we consider three cases based on the choice of the parameter $\theta$. To provide a quantitative comparison, we consider 1,000 experiments with $r=0$ and with the initial condition $x(0)=[x_1(0)~x_2(0)]^\top$, where in each experiment, $x_1(0)$ and $x_2(0)$ are uniformly selected from the interval [-1,1]{{\color{blue}}, which belongs to the RoA shown in Figure \ref{fig:ROAPendulum}}; note that the same initial condition is applied to all cases to ensure a fair comparison.

The obtained results are reported in Table \ref{tab:thetas}, where Performance Index (PI) is defined as $\text{PI}=\sum\left\Vert x(t)\right\Vert$. As seen in this table, using a large $\theta$ can improve the performance with both one-step-ahead predictive control scheme and the NN-based control scheme, even though the improvement with the NN-based control scheme is more significant than that with the one-step-ahead predictive control scheme.

Table \ref{tab:thetas} reveals that for $\theta=0.01$ and $\theta=0.001$, the NN-based control scheme provides a better tracking performance compared to the one-step-ahead predictive control scheme. Note that this is understandable, as the one-step-ahead predictive control scheme is tailored to minimize the cost function given in \eqref{eq:CostFunction}, which is different than the considered metric for comparing the tracking performance.

Figure\ref{fig:NNOPTx1x2} represents a typical time profile of the system states with NN-based control scheme and one-step-ahead predictive control scheme for $\theta=0.0001$. As seen in this figure, the tracking error with the NN-based control is greater than that of the one-step-ahead-predictive control.

{{\color{blue}}Note that NN-based control scheme is faster than the one-step-ahead control scheme, and thus is appropriate for real-time applications. For instance, when $\theta=0.01$, the mean computing time for the NN-based control scheme is 0.565 seconds, while the mean computing time for the one-step-ahead predictive controls scheme is 180.551 seconds; that is the NN-based control scheme is 319 times faster than the one-step-ahead predictive control scheme.}

% {\color{green} Noticed that the optimized approach consume more time on finding optimal solutions. For instance, when $\theta=0.01$, one-step-ahead productive control method cost 180.5508 seconds for 50 iterations. While NN-based control method cost only 0.565 second for same iterations. }

\begin{table}[!t]
    \centering
    \caption{Performance Index for Different Values of $\theta$ and Different Control Methods.}
    \begin{tabular}{c|c|c}
    \hline
     Performance Index & One-Step-Ahead  & NN-Based \\  $[\text{deg}]$ & Predictive Control & Control \\
       \hline\hline
     ${\theta=0.0100}$ & 8.454 & 6.430 \\
       \hline   
     ${\theta=0.0010}$ & 8.495 & 7.047 \\
       \hline  
     ${\theta=0.0001}$ & 8.513 & 11.998\\
       \hline

     % these are original results without rounding.
     % ${\theta=0.0100}$ & 0.14791457599860772 & 0.11224953599434385 \\
     %   \hline   
     % ${\theta=0.0010}$ & 0.14810642505419372 & 0.1229987812689951 \\
     %   \hline  
     % ${\theta=0.0001}$ & 0.14822717718788953 & 0.20941022683157604 \\       
       \end{tabular}
\label{tab:thetas}
\end{table}

\begin{figure}[!t]
    \centering
    \includegraphics[width=8cm]{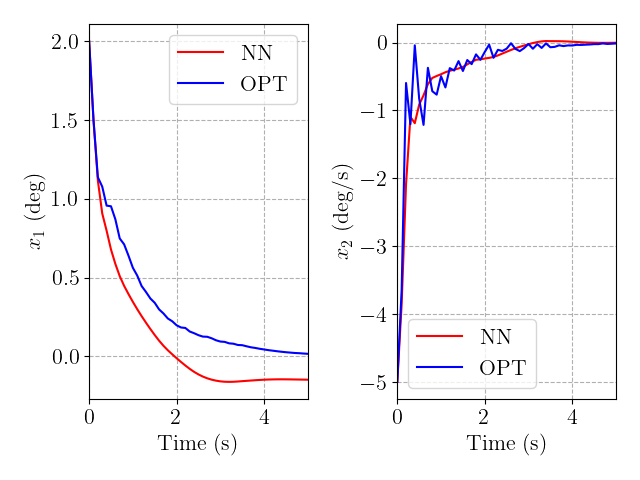}
    \caption{{{\color{blue}}Time profile of $x(t)$ with ${\theta}=0.0001$.}}
    \label{fig:NNOPTx1x2}
\end{figure}

\subsection{Time-Domain Analysis}
This section provides time-domain results of the proposed NN-based control scheme. We present results with $\theta=0.01$, $\theta=0.001$, and $\theta=0.0001$ for $r=0$. {{\color{blue}}For comparison purposes, we also consider the iterative LQR technique, where the dynamics of the inverted pendulum is linearized at every time instant, and then an infinite-horizon LQR control law is computed for the linearized system.}  The initial condition is $x(0)=[-2.3~5]^\top${{\color{blue}}; note that according to Figure\ref{fig:ROAPendulum}, it is obvious that $x(0)\in\Phi(0)$.}

% \textcolor{red}{Can we include the graphs for what the optimial control we trained on would give?}

Figure\ref{fig:ResultsStateInput} shows the time profile of the state variables $x_1(t)$ and $x_2(t)$, and the control input $u(t)$. As seen in this figure, the proposed NN-based control scheme can effectively steer the angle of the inverted pendulum to zero. It should be note that as expected, the larger the value of $\theta$ is, the better the tracking performance can become. In particular, the PI index defined above is 26.37, 26.99, and 31.17 with $\theta=0.01$, $\theta=0.001$, and $\theta=0.0001$, respectively. Similar to the extensive simulation studies reported in Subsection \ref{sec:ExtensiveSimulation}, $\theta=0.01$ and $\theta=0.001$ provide a comparable tracking performance, while $\theta=0.0001$ significantly degrades the tracking performance. 

{{\color{blue}}
Also, Figure\ref{fig:ResultsStateInput} reveals that the iterative LQR technique provides a better solution in comparison with the proposed NN-based control scheme; note that this observation is understandable, as the linearization error $\delta$ discussed in Assumption \ref{assumption:Lipchitz} is small throughout the RoA $\Phi(0)$.}

\begin{figure}[!t]
    \centering
    \includegraphics[width=7cm]{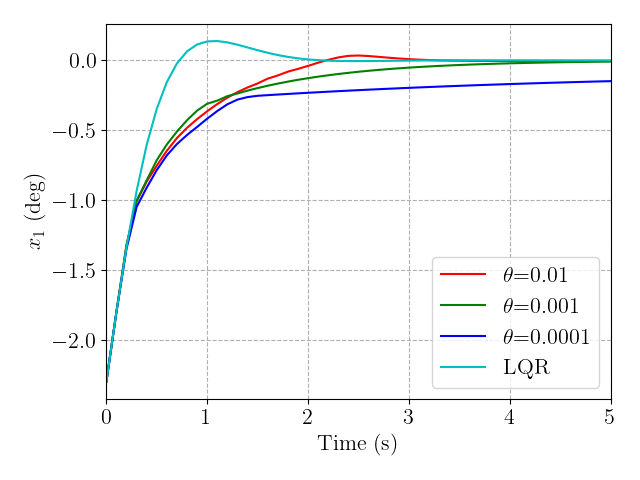}\\
    \includegraphics[width=7cm]{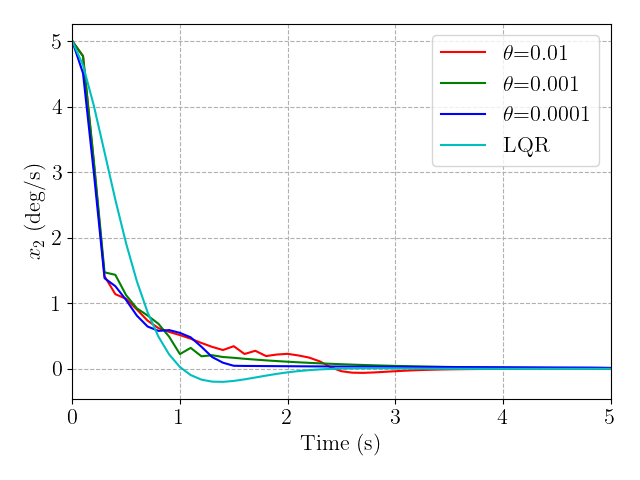}\\
    \includegraphics[width=7cm]{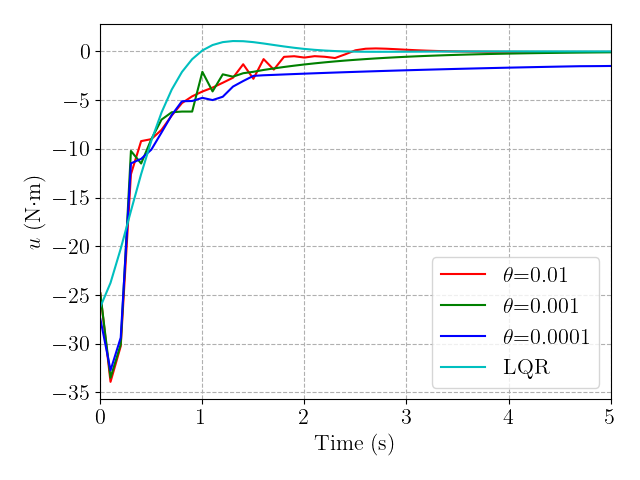}\\
    \caption{{{\color{blue}}Time profile of state vector $x(t)$ and control input $u(t)$.}}
    \label{fig:ResultsStateInput}
\end{figure}

Phase portrait of the system is shown in Figure\ref{fig:x1x2}, where the arrows show the moving direction of states. Note that the length of each arrow shows the speed of movement; that is, the higher the length of the arrow is, the faster the states move.

Figure \ref{fig:deltav} shows the time profile of $\Delta V(t)$ defined in \eqref{eq:DV1}. As seen in this figure, when the state vector $x(t)$ is far away from the equilibrium point (i.e., the origin), $\Delta V(t)$ is negative implying that the Lyapunov function is decreasing. As $x(t)$ gets closer to the equilibrium point, due to training errors, $\Delta V(t)$ may take positive values. Nevertheless, as mentioned in Theorem \ref{theorem:NNControl}, the system is stable and the tracking error remains bounded.

% $\theta=0.01$ is 0.472, PI is 0.471 for $\theta=0.001$ , and 0.544 for $\theta=0.0001$, respectively.

% difference of the Lyapunov function $\Delta V$ in the motion of Figure\ref{fig:x1x2}. We noticed the differences are always less than 0 except limit moments. {\color{red}It is reasonable, we know $\left\Vert P \right\Vert$ increases as state closes to 0. And it leads to the NN model prediction error increases.(We need to discuss this, should I add dataset figures for $\bar{x_r}=0$?)}

% {\color{red}Continue with phase portrait and $\Delta V(t)$ graphs ...}

\begin{figure}[!t]
    \centering
    \includegraphics[width=7cm]{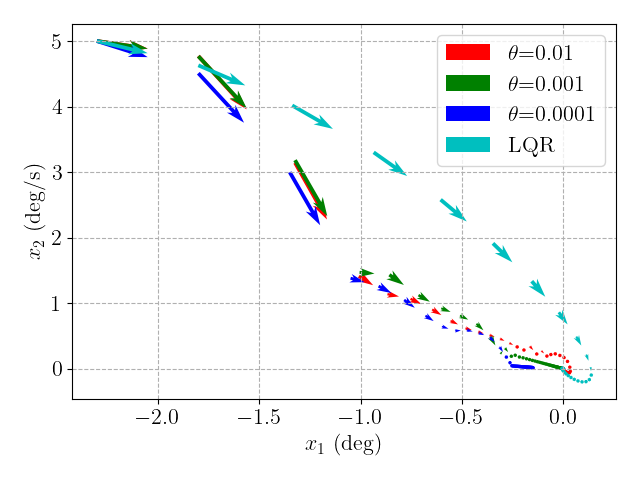}\\
    \caption{{{\color{blue}}Phase portrait graph.}}
    \label{fig:x1x2}
\end{figure}

\begin{figure}[!t]
    \centering
\includegraphics[width=7cm]{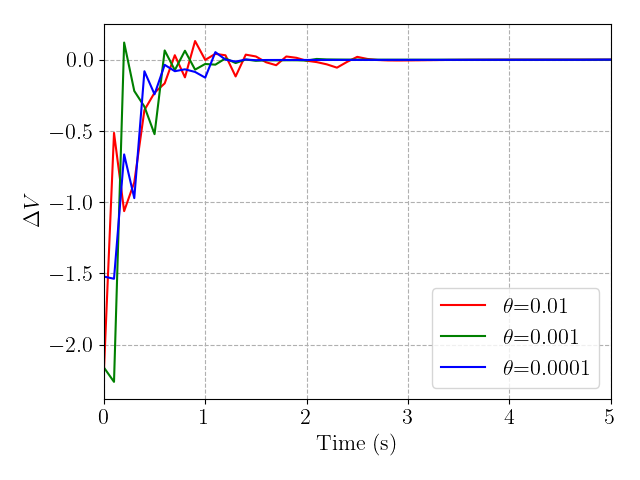}\\
    \caption{{{\color{blue}}Time profile of $\Delta V(t)$.}}
    \label{fig:deltav}
\end{figure}

% \begin{figure}[!t]
%     \centering
%     \includegraphics[width=0.8\columnwidth]{images/x2.png}\\
%     \caption{$x_2$ values in time-domain when $\theta=0.01$ for the inverted pendulum system.}
%     \label{fig:x2_time}
% \end{figure}

% \begin{figure}[!t]
%     \centering
%     \includegraphics[width=0.8\columnwidth]{images/ut.png}
%     \caption{$u$ values in time-domain when $\theta=0.01$ for the inverted pendulum system.}
%     \label{fig:u_time}
% \end{figure}

{{\color{blue}}
\section{Experimental Results}
To validate the proposed scheme in in practical scenarios, we employ it to control the position of a Parrot Bebop 2 drone. The dynamical model of the drone can be expressed as \cite{AmiriMECC}:
$$
\begin{aligned}
& \dot{x}=\left[\begin{array}{cccccc}
0 & 1 & 0 & 0 & 0 & 0 \\
0 & -0.0527 & 0 & 0 & 0 & 0 \\
0 & 0 & 0 & 1 & 0 & 0 \\
0 & 0 & 0 & -0.0187 & 0 & 0 \\
0 & 0 & 0 & 0 & 0 & 1 \\
0 & 0 & 0 & 0 & 0 & -1.7873
\end{array}\right] x+\left[\begin{array}{ccc}
0 & 0 & 0 \\
-5.4779 & 0 & 0 \\
0 & 0 & 0 \\
0 & -7.0608 & 0 \\
0 & 0 & 0 \\
0 & 0 & -1.7382
\end{array}\right] u, \\
& y=\left[\begin{array}{llllll}
1 & 0 & 0 & 0 & 0 & 0 \\
0 & 0 & 1 & 0 & 0 & 0 \\
0 & 0 & 0 & 0 & 1 & 0
\end{array}\right] x,
\end{aligned}
$$

% \begin{table}
%     \centering
%     \caption{{{\color{blue}}Identified parameters for the Parrot Bebop 2 drone.}}
%     \begin{tabular}{c|c|c|c|c|c|c}
%     \hline
% \text { Parameter } & $\alpha_x$ & $\alpha_y$ & $\alpha_z$ & $\beta_x$ & $\beta_y$ & $\beta_z$ \\
% \hline \text { Value } & 0.0527 & 0.0187 & 1.7873 & -5.4779 & -7.0608 & -1.7382 \\
% \hline  
%     \end{tabular}
%     \label{tab:drone}
% \end{table}

% The parameters in the matrices above are detailed in Table \ref{tab:drone}. 

It should be noted that the dynamics along X, Y, and Z directions are decoupled; thus, we employed identical structured NN models in all three directions. More precisely, three feedforward NNs are utilized, each with 6 hidden
layers, and 8, 32, 64, 64, 32, and 16 neurons in the hidden layers. 

To collect the dataset, for the X and Y directions, we set $\mathcal{X}=[-0.5,0.5]\times[-1,1]$ and $\mathcal{R}=[0,0]$. For the Z direction, we set $\mathcal{X}=[1,2]\times[-1,1]$ and $\mathcal{R}=[1.5,0]$. The training dataset, denoted as $\mathbb{D}_{\text{training}}$, is generated by gridding the aforementioned regions with steps of 0.01; thus, $|\mathbb{D}_{\text{training}}|=40404$ for each direction.

We utilize the YALMIP toolbox to solve the optimization problem \eqref{eq:OptimizationProblemMain} for each data point, where $Q_x=20I_2$ and $Q_u=0.1$. We set $\theta=1$ for all scenarios. After collecting the data, we used the same training parameters and strategies as those used for the inverted pendulum system within 10000 epochs. Subsequently, the models are trained and converted into executable MATLAB format.

\begin{figure}[!t]
    \centering
    \includegraphics[width=9cm]{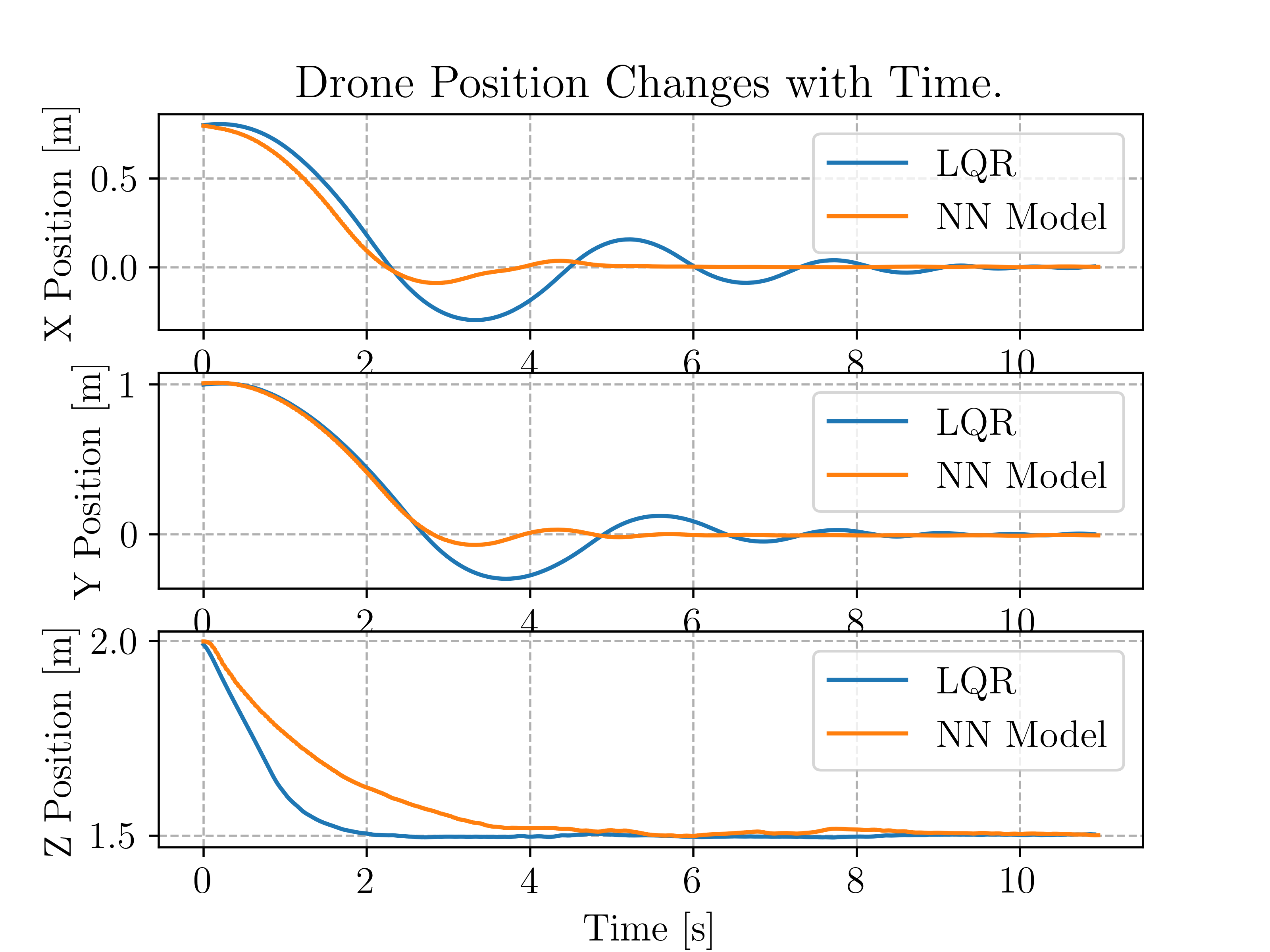}
    \caption{{{\color{blue}}Time profile of the drone's position.}}
    \label{fig:Drone_cpr}
\end{figure}

We compared the trained NN models with the LQR control method. To ensure a fair comparison, we set the weighting matrices of the LQR to be the same as the weighting matrices in the optimization problem \eqref{eq:OptimizationProblemMain}. 

Figure \ref{fig:Drone_cpr} presents the time profile of the drone's position. Also, control effort defined as $\sum\left\vert u(t)\right\vert$ for the proposed method and LQR is reported in Table \ref{tab:drone_effort}. As seen in Figure \ref{fig:Drone_cpr}, for the X and Y directions, our NN models outperformed the LQR method, as the linearization error in these directions is large; while the LQR method provided relatively better results in the Z direction. This observation is reasonable, as the optimization problem \eqref{eq:OptimizationProblemMain} is different that underlying optimization problem in LQR framework. 
}

% our proposed method optimizes one-step-ahead behavior, while LQR considers the optimal solution in an infinite-horizon cost function. Another reason for this difference lies in our proposed optimization scheme, which includes constraints that eliminate some local minima present in the LQR method. This may explain why our method requires less control effort than LQR, as shown in .

\section{Conclusion}\label{sec:Con}
This paper proposed a systematic and comprehensive methodology to design provably-stable NN-based control schemes for affine nonlinear systems. First,  a novel one-step-ahead predictive control method was developed; its stability and convergence properties were analytically proven. Then, an approach was presented to train a NN that imitates the behavior of the one-step-ahead predictive control scheme in a given operating region. Stability and convergence properties of the closed-loop system with the trained NN in the loop were shown via rigorous analysis. In particular, it was shown that the resulting NN-based control scheme guarantees that the states of the system remain bounded, and asymptotically converge to a neighborhood around the desired equilibrium point, with a tunable proximity threshold. The effectiveness of the proposed approach was assessed via extensive simulation and experimental studies.

{{\color{blue}}The main limitation of the proposed method is that the optimization problem \eqref{eq:OptimizationProblemMain} is nonlinear and non-convex with respect to decision variables $u$ and $P$, and existing solvers may not able to solve its all instances for any given system. Thus, the NN trained on the dataset obtained by numerically solving the optimization problem \eqref{eq:OptimizationProblemMain} may give a large approximation error (i.e., large $\bar{\Delta u}$ and $\bar{\Delta P}$), and thus according to \eqref{eq:vartheta}, the resulting NN-based control scheme may yield poor performance. Future work will investigate methods to effectively solve the optimization problem \eqref{eq:OptimizationProblemMain} and collect a decent training dataset.}

\begin{table}[!t]
    \centering
    \caption{{{\color{blue}}Control effort (i.e., $\sum\left\vert u(t)\right\vert$) with NN models and LQR in three directions.}}
    \begin{tabular}{c|c|c|c}
    \hline
\text { Direction } & X & Y & Z \\
\hline
\text { Proposed Method } & 33.3 & 29.6 & 65.6 \\
\hline \text { LQR } & 61.7 & 59.8 & 95.9\\
\hline  
    \end{tabular}
    \label{tab:drone_effort}
\end{table}

\section*{Data Availability}
All the cade and data used in simulation and experimental studies are available at \url{https://github.com/anran-github/Provably-Stable-Neural-Network-Based-Control-of-Nonlinear-Systems.git}. 

\bibliography{references} 
\bibliographystyle{IEEEtran}

\end{document}